\documentclass[oneside,english,reqno]{amsart}
\usepackage[T1]{fontenc}
\usepackage[latin9]{inputenc}
\usepackage{geometry}
\usepackage{xcolor}
\usepackage{pdfcolmk}
\usepackage{amsthm}
\usepackage{amsbsy}
\usepackage{amssymb}
\PassOptionsToPackage{normalem}{ulem}
\usepackage{ulem}

\makeatletter

\providecolor{lyxadded}{rgb}{0,0,1}
\providecolor{lyxdeleted}{rgb}{1,0,0}
\numberwithin{equation}{section}
\numberwithin{figure}{section}
  \theoremstyle{plain}
  \newtheorem*{thm*}{\protect\theoremname}
  \theoremstyle{definition}
  \newtheorem*{example*}{\protect\examplename}
\theoremstyle{plain}
\newtheorem{thm}{\protect\theoremname}
  \theoremstyle{plain}
  \newtheorem{prop}[thm]{\protect\propositionname}
  \theoremstyle{plain}
  \newtheorem*{fact*}{\protect\factname}

\DeclareMathOperator{\sgn}{sgn}

\allowdisplaybreaks[2]
\theoremstyle{remark}
\newtheorem*{rems*}{Remarks}

\makeatother

\usepackage{babel}
  \providecommand{\examplename}{Example}
  \providecommand{\factname}{Fact}
  \providecommand{\propositionname}{Proposition}
  \providecommand{\theoremname}{Theorem}
\providecommand{\theoremname}{Theorem}

\begin{document}

\title{On the Fourier Expansion of Word Maps}

\author{Ori Parzanchevski and Gili Schul}

\address{{\small Einstein Institute of Mathematics, Hebrew University, Jerusalem}}

\email{parzan@math.huji.ac.il, gili.schul@mail.huji.ac.il}
\begin{abstract}
Frobenius observed that the number of times an element of a finite
group is obtained as a commutator is given by a specific combination
of the irreducible characters of the group. More generally, for any
word $w$ the number of times an element is obtained by substitution
in $w$ is a class function. Thus, it has a presentation as a combination
of irreducible characters, called its Fourier expansion. In this paper
we present formulas regarding the Fourier expansion of words in which
some letters appear twice. These formulas give simple proofs for classical
results, as well as new ones.
\end{abstract}
\maketitle

\section{Introduction}

Let $w$ be a word in $\mathbf{F}_{d}=\mathbf{F}\left(x_{1},\ldots,x_{d}\right)$,
the free group on $d$ generators. Given a finite group $G$, $w$
induces a \emph{word map }from\emph{ }$G^{d}=G\times\cdots\times G$
to $G$, defined by substitution: if $w=\prod_{j=1}^{r}x_{i_{j}}^{\varepsilon_{j}}$
(where $\varepsilon_{j}=\pm1$), then the word map $w:G^{d}\rightarrow G$
is defined by $w\left(g_{1},\ldots,g_{d}\right)=\prod_{j=1}^{r}g_{i_{j}}^{\varepsilon j}$.
For example, for $w=\left[x_{1},x_{2}\right]=x_{1}x_{2}x_{1}^{-1}x_{2}^{-1}$,
$w\left(g,h\right)=ghg^{-1}h^{-1}$. It is a classical question to
understand the distribution which $w$ induces on $G$, i.e., the
function 
\[
N_{w}\left(g\right)=\left|w^{-1}\left(g\right)\right|=\sharp\left\{ \left(g_{1},\ldots,g_{d}\right)\in G^{d}|\ w\left(g_{1},\ldots,g_{d}\right)=g\right\} 
\]
(some authors write $N_{w,G}$ to emphasis that this depends on $G$).
The simple observation that $N_{w}$ is a class function shows that
one can write 
\[
N_{w}=\sum\limits _{\chi\in\mathrm{Irr}\left(G\right)}N_{w}^{\chi}\cdot\chi
\]
for unique $N_{w}^{\chi}\in\mathbb{C}$, where $\mathrm{Irr}\left(G\right)$
are the irreducible characters of $G$. The coefficients $N_{w}^{\chi}$
are called the \emph{Fourier coefficients }of $w$ with respect to
$G$, and their study goes back to Frobenius \cite{frobenius1896gruppencharaktere},
who showed that $N_{\left[x_{1},x_{2}\right]}^{\chi}=\frac{\left|G\right|}{\chi\left(1\right)}$
for any $G$. Since $\chi\left(1\right)$ divides $\left|G\right|$,
this shows in particular that $N_{\left[x_{1},x_{2}\right]}$ is itself
a character of $G$. Several authors have noticed that Frobenius'
argument generalizes to $N_{\left[x_{1},y_{1}\right]\cdot\ldots\cdot\left[x_{k},y_{k}\right]}^{\chi}=\left(\frac{\left|G\right|}{\chi\left(1\right)}\right)^{2k-1}$,
and the Fourier expansions of other words were studied as well (see
\cite[chapter 4]{isaacs1994character}, and more recent results in
\cite{tambour2000number,AV10}).\medskip{}

If $w$ and $w'$ are equivalent under $\mathrm{Aut}\left(\mathbf{F}_{d}\right)$,
then $N_{w}\equiv N_{w'}$ for any $G$, so that one can always replace
$w\in\mathbf{F}_{d}$ by $\varphi\left(w\right)$ for any $\varphi\in\mathrm{Aut}\left(\mathbf{F}_{d}\right)$.
If some letter $x$ appears only once in $w$ (either as $x$ or as
$x^{-1}$, but not both), then $w$ is equivalent under $\mathrm{Aut}\left(\mathbf{F}_{d}\right)$
to the single letter $x$, and thus $N_{w}$ is constant. In this
paper we consider the next simplest case, in which some letter appears
twice in $w$, either with the same sign or with different signs.

We say that $x$ is a \emph{square }in $w$ if $x$ appears twice
in it, and $x^{-1}$ does not, or vice versa.\emph{ }We say that $x$
is \emph{dismissible} in $w$ if both $x$ and $x^{-1}$ appear in
$w$, each of them only once. In this paper we present formulas regarding
the Fourier coefficients of words with at least one square or dismissible
letter. The results in \cite{frobenius1896gruppencharaktere,tambour2000number,AV10}
are obtained as special cases of these formulas, and further cases
and applications are shown in Section \ref{sec:Applications-of-the}.

\subsection{Statement of the results}

It is well known that for any word $w\in\mathbf{F}_{d}$ and finite
group $G$, the Fourier coefficients of $w$ with respect to $G$
can be obtained by summation over all substitutions of $G$'s elements
in $w$'s letters:
\begin{equation}
N_{w}^{\chi}=\frac{1}{\left|G\right|}\sum_{\overline{x}\in G^{d}}\overline{\chi}\left(w\left(x_{1},\ldots,x_{d}\right)\right),\label{eq:inversion-formula}
\end{equation}
where $\overline{x}$ stands for $\left(x_{1},\ldots,x_{d}\right)$
(for the proof see Proposition \ref{prop:misc}). We will show the
following:
\begin{thm*}
Let $w$ be a word in \textup{$\mathbf{F}_{d}$, $G$ a finite group,
and $\chi\in\mathrm{Irr}\left(G\right)$.}
\begin{enumerate}
\item There is a formula for $N_{w}^{\chi}$ which involves summation only
over the non-dismissible letters of $w$. 
\item There is a formula for $N_{w}^{\chi}$ which involves the \emph{Frobenius-Schur
indicator }of $\chi$ (\ref{eq:FS_indicator}), and summation only
over the non-dismissible and non-square letters of $w$. 
\end{enumerate}
\end{thm*}
The Frobenius-Schur indicator of $\chi$ is 
\begin{equation}
\mathcal{FS}_{\chi}=\begin{cases}
1 & \chi\:\mathrm{is\: afforded\: by\: a\: real\: representation}\\
-1 & \chi\:\mathrm{is\: real\: but\: not\: afforded\: by\: a\: real\: representation}\\
0 & \chi\:\mathrm{is\: not\: real,}
\end{cases}\label{eq:FS_indicator}
\end{equation}
and it is a classical fact that $\mathcal{FS}_{\chi}=N_{x^{2}}^{\chi}$
(see e.g. \cite[Proposition 39]{serre1977linear}).

The formulas obtained are the following:
\begin{enumerate}
\item \textbf{Dismissible letters:} Let $w\in\mathbf{F}\left(x_{1},\ldots,x_{d},y_{1},\ldots,y_{n}\right)$
be a word in which $y_{1},\ldots,y_{n}$ are dismissible letters.
Define words $w_{0},\ldots,w_{2n}\in\mathbf{F}\left(x_{1},\ldots,x_{d}\right)$
by writing $w$ as 
\[
w=w_{0}\negmedspace\cdot\negmedspace y_{i_{1}}^{\pm1}\negmedspace\cdot\negmedspace w_{1}\negmedspace\cdot\negmedspace y_{i_{2}}^{\pm1}\!\cdot\!\ldots\!\cdot\! w_{2n-1}\negmedspace\cdot\negmedspace y_{i_{2n}}^{\pm1}\negmedspace\cdot\negmedspace w_{2n},
\]
where any $w_{j}$ may be the empty word $1$. We define words $W_{i}\in\mathbf{F}\left(x_{1},\ldots,x_{d}\right)$
as follows. For $W_{1}$, start reading $w$ from $w_{0}$, and upon
reaching a dismissible letter, jump to its inverse, and continue reading
from there. When reaching the end of $w$, stop reading - the word
obtained so far is $W_{1}$. Now, find the first $w_{j}$ which has
not been read yet, and start reading $W_{2}$ from it, again jumping
from dismissible letters to their inverse. You finish reading $W_{2}$
upon reaching your starting point. Continuing in this manner, one
obtains $W_{1},\ldots,W_{r}$ (where $r=r\left(w\right)$) - the process
stops when all $w_{j}$ have been read. Our formula then states that
the Fourier expansion of $w$ is given by 
\begin{equation}
N_{w}^{\chi}=\frac{\left|G\right|^{n-1}}{\chi\left(1\right)^{n}}\sum_{\overline{x}\in G^{d}}\overline{\chi}\left(W_{1}\left(\overline{x}\right)\right)\cdot\ldots\cdot\overline{\chi}\left(W_{r}\left(\overline{x}\right)\right).\label{eq:formula}
\end{equation}
It may happen that $d=0$ (if all letters are dismissible), in which
case the summation is over the empty substitution, and 
\begin{equation}
N_{w}^{\chi}=\frac{\left|G\right|^{n-1}}{\chi\left(1\right)^{n-r}}.\label{eq:fomula_empty}
\end{equation}

\item \textbf{Squares: }Let $w=w_{1}\negmedspace\cdot\negmedspace y\negmedspace\cdot\negmedspace w_{2}\negmedspace\cdot\negmedspace y\negmedspace\cdot\negmedspace w_{3}\in\mathbf{F}\left(x_{1},\ldots,x_{d},y\right)$
where $y$ does not appear in $w_{1},w_{2},w_{3}\in\mathbf{F}\left(x_{1},\ldots,x_{d}\right)$.
Then
\[
N_{w}^{\chi}=\frac{\left|G\right|}{\chi\left(1\right)}\mathcal{FS}_{\chi}N_{w_{1}w_{2}^{-1}w_{3}}^{\chi},
\]
where in $N_{w_{1}w_{2}^{-1}w_{3}}^{\chi}$ we refer to $w_{1}w_{2}^{-1}w_{3}$
as a word in \textbf{$\mathbf{F}\left(x_{1},\ldots,x_{d}\right)$.}\end{enumerate}
\begin{example*}
To demonstrate the formula for dismissible letters, let 
\[
w=\underbrace{x_{1}}_{w_{0}}y_{1}\underbrace{x_{1}x_{2}}_{w_{1}}y_{3}\underbrace{x_{2}x_{1}}_{w_{2}}y_{1}^{-1}\underbrace{x_{1}^{3}}_{w_{3}}y_{2}\underbrace{x_{3}^{-1}}_{w_{4}}y_{3}^{-1}\underbrace{x_{3}^{2}}_{w_{5}}y_{2}^{-1}\underbrace{x_{3}}_{w_{6}}.
\]
We obtain 
\begin{align*}
W_{1} & =w_{0}w_{3}w_{6}=x_{1}^{4}x_{3}\\
W_{2} & =w_{1}w_{5}w_{4}w_{2}=x_{1}x_{2}x_{3}x_{2}x_{1}.
\end{align*}
and thus
\[
N_{w}^{\chi}=\frac{\left|G\right|^{2}}{\chi\left(1\right)^{3}}\sum_{x_{1},x_{2},x_{3}\in G}\overline{\chi}\left(x_{1}^{4}x_{3}\right)\overline{\chi}\left(x_{1}x_{2}x_{3}x_{2}x_{1}\right).
\]

\medskip{}

In the next section we demonstrate some applications of the formulas,
including the results referred to earlier. Section \ref{sec:Preliminaries}
presents some preparatory results for the main theorems, which are
proved in Section \ref{sec:Proof-of-the}.
\end{example*}
\begin{rems*}$ $
\begin{enumerate}
\item The ambient free group of $w$ affects $N_{w}$. For example, for
$w=xy\in F\left(x,y\right)$ we have $N_{w}\equiv\left|G\right|$
(for every finite $G$), but for $w=xy\in F\left(x,y,z\right)$ we
have $N_{w}\equiv\left|G\right|^{2}$. 
\item For $1$ the empty word on zero letters, we have $N_{1}\left(g\right)=\delta_{g,e}$
and thus $N_{1}^{\chi}=\frac{\chi\left(1\right)}{\left|G\right|}$.
\item In order to keep things tidy we shall occasionally write $w$ for
$w\left(x_{1},\ldots,x_{d}\right)$, where $x_{1},\ldots,x_{d}$ are
elements of $G$. For example, we write (\ref{eq:formula}) as 
\[
N_{w}^{\chi}=\frac{\left|G\right|^{n-1}}{\chi\left(1\right)^{n}}\sum_{\overline{x}\in G^{d}}\overline{\chi}\left(W_{1}\right)\cdot\ldots\cdot\overline{\chi}\left(W_{r}\right).
\]

\end{enumerate}
\end{rems*}\medskip{}

\subsection*{Acknowledgements }

We would like to thank our advisors Aner Shalev and Alex Lubotzky
for their support, and Doron Puder and Zlil Sela for valuable discussions.
This research was supported by an ERC Advanced Grant.

\section{\label{sec:Applications-of-the}Corollaries and applications}

We give now examples of applications and corollaries of the main formulas. 
\begin{enumerate}
\item \label{enu:Frobenius-formula}Frobenius' formula: for the commutator
\[
w=\left[y_{1},y_{2}\right]=y_{1}y_{2}y_{1}^{-1}y_{2}^{-1}=w_{0}y_{1}w_{1}y_{2}w_{2}y_{1}^{-1}w_{3}y_{2}^{-1}w_{4}
\]
we have $d=0$, $n=2$, and $w_{0},\ldots,w_{4}=1$. This gives $W_{1}=w_{0}w_{3}w_{2}w_{1}w_{4}=1$,
so that $r=1$, and thus by (\ref{eq:fomula_empty}) 
\[
N_{\left[y_{1},y_{2}\right]}^{\chi}=\frac{\left|G\right|}{\chi\left(1\right)}.
\]
More generally, $w=\left[y_{1},y_{2}\right]\cdot\ldots\cdot\left[y_{2k-1},y_{2k}\right]$
gives $d=0$, $n=2k$, $r=1$, and $W_{1}=1$, hence 
\[
N_{\left[y_{1},y_{2}\right]\cdot\ldots\cdot\left[y_{2k-1},y_{2k}\right]}^{\chi}=\left(\frac{\left|G\right|}{\chi\left(1\right)}\right)^{2k-1}
\]
(which can also be deduced from (\ref{eq:conv-disj})).
\item Let $w_{0},w_{1}$ be words in $\mathbf{F}\left(x_{1},\ldots,x_{d}\right)$
(not necessarily disjoint), and observe $w=w_{0}yw_{1}y^{-1}$. In
this case $W_{1}=w_{0}$ and $W_{2}=w_{1}$, which gives 
\[
N_{w_{0}yw_{1}y^{-1}}^{\chi}=\frac{1}{\chi\left(1\right)}\sum_{\overline{x}\in G^{d}}\overline{\chi}\left(w_{0}\right)\overline{\chi}\left(w_{1}\right).
\]

\item \label{enu:wyw-1y-1}For $w\in\mathbf{F}_{d}$, taking $w_{0}=w$,
$w_{1}=w^{-1}$ in the previous example gives 
\begin{align}
N_{\left[w,y\right]}^{\chi} & =\frac{1}{\chi\left(1\right)}\sum_{\overline{x}\in G^{d}}\overline{\chi}\left(w\right)\chi\left(w\right)=\frac{1}{\left|\chi\left(1\right)\right|}\sum_{g\in G}N_{w}\left(g\right)\overline{\chi}\left(g\right)\chi\left(g\right)\nonumber \\
 & =\frac{\left|G\right|}{\chi\left(1\right)}\left\langle N_{w}\chi,\chi\right\rangle .\label{eq:[w,y]-from-formula}
\end{align}
In terms of the Fourier expansion of $N_{w}$, we obtain 
\[
N_{\left[w,y\right]}^{\chi}=\frac{\left|G\right|}{\chi\left(1\right)}\sum_{\psi\in\mathrm{Irr}\left(G\right)}\left\langle \psi\chi,\chi\right\rangle N_{w}^{\psi},
\]
which is (roughly) Lemma 4.6 in \cite{AV10}. We shall return to this
example in Proposition \ref{prop:[w,z]}.
\item Similarly, for $w_{0}=w_{1}=w$ we obtain 
\begin{equation}
N_{wywy^{-1}}^{\chi}=\frac{1}{\left|\chi\left(1\right)\right|}\sum_{g\in G}N_{w}\left(g\right)\chi\left(g\right)^{2}=\frac{\left|G\right|}{\chi\left(1\right)}\left\langle N_{w}\chi,\overline{\chi}\right\rangle .\label{eq:wywy^-1}
\end{equation}

\item If $W_{1},\ldots,W_{r}=1$, then 
\begin{equation}
N_{w}^{\chi}=\frac{\left|G\right|^{n-1}}{\chi\left(1\right)^{n}}\sum_{\overline{x}\in G^{d}}\overline{\chi}\left(W_{1}\right)\cdot\ldots\cdot\overline{\chi}\left(W_{r}\right)=\frac{\left|G\right|^{n+d-1}}{\chi\left(1\right)^{n-r}}\label{eq:admissible}
\end{equation}
and the only question is what is $r$. In this case one can show that
$w$ is equivalent under $\mathrm{Aut}\left(\mathbf{F}_{d}\right)$
to a word $w'$ in which all letters are dismissible (such a word
is called \emph{admissible}). It is a classical observation that admissible
words describe a relation defining the fundamental group of an orientable
surface. Thus, $w'$ is equivalent under $\mathrm{Aut}\left(\mathbf{F}_{d}\right)$
to a product of $g$ disjoint commutators, where $g$ is the genus
of the surface it defines. Finding $r$ amounts to finding this genus,
as $g=\frac{n-r+1}{2}$.
\item For example, for the admissible words $w=y_{1}y_{2}\ldots y_{n}y_{1}^{-1}y_{2}^{-1}\ldots y_{n}^{-1}$
we obtain a result from \cite{tambour2000number}: for $n$ odd we
have 
\begin{align*}
W_{1} & =w_{0}w_{n+1}w_{2}w_{n+3}w_{4}\ldots w_{n-1}w_{2n}=1\\
W_{2} & =w_{1}w_{n+2}w_{3}w_{n+4}\ldots w_{n-2}w_{2n-1}w_{n}=1
\end{align*}
so that $r=2$ and $N_{w}^{\chi}=\frac{\left|G\right|^{n-1}}{\chi\left(1\right)^{n-2}}.$
For $n$ even, 
\[
W_{1}=w_{0}w_{n+1}w_{2}w_{n+3}w_{4}\ldots w_{2n-1}w_{n}w_{1}w_{n+2}w_{3}w_{n+4}\ldots w_{n-2}w_{2n-1}w_{n}=1
\]
covers all $w_{i}$, hence $r=1$ and $N_{w}^{\chi}=\frac{\left|G\right|^{n-1}}{\chi\left(1\right)^{n-1}}$.
\item Combining (\ref{enu:wyw-1y-1}) and (\ref{enu:Frobenius-formula})
gives a formula for $w=\left[\left[x,y\right],z\right]$:
\[
N_{\left[\left[x,y\right],z\right]}^{\chi}=\frac{\left|G\right|}{\chi\left(1\right)}\sum_{\psi\in\mathrm{Irr}\left(G\right)}\left\langle \psi\chi,\chi\right\rangle N_{\left[x,y\right]}^{\psi}=\frac{\left|G\right|^{2}}{\chi\left(1\right)}\sum\limits _{\psi\in\mathrm{Irr}\left(G\right)}\frac{\left\langle \psi\chi,\chi\right\rangle }{\psi\left(1\right)}.
\]
This formula was already obtained, using different methods, by Amit
\& Vishne \cite{AV10}. In particular it implies that knowing the
decomposition of tensor product of irreducible representation of $G$
gives the Fourier expansion of $\left[\left[x,y\right],z\right]$.%
\footnote{If this decomposition is given by $\chi_{i}\cdot\chi_{j}=\sum c_{i,j,k}\cdot\chi_{k}$,
then $c_{i,j,k}$ are the \emph{Clebsch\textendash{}Gordan coefficients
}of $G$. For $G=S_{n}$ they are also referred to as \emph{Kronecker
coefficients}, and for $G=\mathrm{GL}_{n}$ as \emph{Littlewood\textendash{}Richardson
coefficients}.%
} Continuing by induction gives the coefficients of $\left[\left[\left[x_{1},x_{2}\right],x_{3}\right],\ldots,x_{n}\right]$,
which can be presented in terms of matrix exponentiation (see \cite{AV10}).
\item Nested and non-nested squares: if $w=w_{1}\negmedspace\cdot\negmedspace x\negmedspace\cdot\negmedspace w_{2}\negmedspace\cdot\negmedspace x\negmedspace\cdot\negmedspace w_{3}\negmedspace\cdot\negmedspace y\negmedspace\cdot\negmedspace w_{4}\negmedspace\cdot\negmedspace y\negmedspace\cdot\negmedspace w_{5}$
where the $w_{i}$ do not contain $x$ and $y$ (but are not necessarily
disjoint), then 
\begin{align*}
N_{w}^{\chi} & =\frac{\left|G\right|}{\chi\left(1\right)}\cdot\mathcal{FS}_{\chi}\cdot N_{w_{1}xw_{2}xw_{3}w_{4}^{-1}w_{5}}^{\chi}=\frac{\left|G\right|^{2}}{\chi\left(1\right)^{2}}\cdot\mathcal{FS}_{\chi}^{2}\cdot N_{w_{1}w_{2}^{-1}w_{3}w_{4}^{-1}w_{5}}^{\chi}\\
 & =\frac{\left|G\right|^{2}}{\chi\left(1\right)^{2}}\cdot\delta_{\chi\in\mathbb{R}}\cdot N_{w_{1}w_{2}^{-1}w_{3}w_{4}^{-1}w_{5}}^{\chi}
\end{align*}
where $\delta_{\chi\in\mathbb{R}}$ is one if $\chi$ is real and
zero otherwise. For $w'=w_{1}\negmedspace\cdot\negmedspace x\negmedspace\cdot\negmedspace w_{2}\negmedspace\cdot\negmedspace y\negmedspace\cdot\negmedspace w_{3}\negmedspace\cdot\negmedspace x\negmedspace\cdot\negmedspace w_{4}\negmedspace\cdot\negmedspace y\negmedspace\cdot\negmedspace w_{5}$,
however, we obtain 
\[
N_{w'}^{\chi}=\frac{\left|G\right|}{\chi\left(1\right)}\cdot\mathcal{FS}_{\chi}\cdot N_{w_{1}xw_{2}w_{4}^{-1}x^{-1}w_{3}^{-1}w_{5}}^{\chi}
\]
in which $x$ became dismissible, and by (\ref{eq:formula})
\[
N_{w'}^{\chi}=\frac{\left|G\right|}{\chi\left(1\right)^{2}}\cdot\mathcal{FS}_{\chi}\sum_{\overline{x}\in G^{d}}\overline{\chi}\left(w_{1}w_{3}^{-1}w_{5}\right)\overline{\chi}\left(w_{2}w_{4}^{-1}\right).
\]
Had we taken $w''=w_{1}\negmedspace\cdot\negmedspace x\negmedspace\cdot\negmedspace w_{2}\negmedspace\cdot\negmedspace y\negmedspace\cdot\negmedspace w_{3}\negmedspace\cdot\negmedspace x^{-1}\negmedspace\cdot\negmedspace w_{4}\negmedspace\cdot\negmedspace y\negmedspace\cdot\negmedspace w_{5}$,
we would get 
\[
N_{w''}^{\chi}=\frac{\left|G\right|}{\chi\left(1\right)}\cdot\mathcal{FS}_{\chi}\cdot N_{w_{1}xw_{2}w_{4}^{-1}xw_{3}^{-1}w_{5}}^{\chi},
\]
i.e. $x$ became a square, and now
\begin{equation}
N_{w''}^{\chi}=\frac{\left|G\right|^{2}}{\chi\left(1\right)^{2}}\cdot\delta_{\chi\in\mathbb{R}}\cdot N_{w_{1}w_{4}w_{2}^{-1}w_{3}^{-1}w_{5}}^{\chi}.\label{eq:square_dismissible}
\end{equation}
Notice that had we applied (\ref{eq:formula}) to $w''$, we would
have get 
\[
N_{w''}^{\chi}=\frac{1}{\chi\left(1\right)}\sum_{{\overline{x}\in G^{d}\atop y\in G}}\overline{\chi}\left(w_{1}w_{4}yw_{5}\right)\overline{\chi}\left(w_{2}yw_{3}\right)
\]
which requires summation over more substitutions then (\ref{eq:square_dismissible}).
In general, squares are better taken care of before dismissible letters.
\item Denote $\left\{ x,y\right\} \overset{{\scriptscriptstyle def}}{=}xyxy^{-1}$.
By (\ref{eq:wywy^-1}), or by (\ref{eq:square_dismissible}), we have
$N_{\left\{ x,y\right\} }^{\chi}=\frac{\left|G\right|}{\chi\left(1\right)}\cdot\delta_{\chi\in\mathbb{R}}$.
Comparing this with $N_{\left[x,y\right]}^{\chi}=\frac{\left|G\right|}{\chi\left(1\right)}$
we see that $N_{\left\{ x,y\right\} }\equiv N_{\left[x,y\right]}$
iff all character of $G$ are real. This happens iff every element
of $g$ is conjugate to its inverse, and it is an easy exercise that
this in itself is equivalent to $N_{\left\{ x,y\right\} }\equiv N_{\left[x,y\right]}$.
However, when considering $w=\prod_{i=1}^{k}\left[x_{i},y_{i}\right]$
and $w'=\prod_{i=1}^{k}\left\{ x_{i},y_{i}\right\} $, we have 
\[
N_{w}^{\chi}=\left(\frac{\left|G\right|}{\chi\left(1\right)}\right)^{2k-1}\qquad\mathrm{and}\qquad N_{w'}^{\chi}=\left(\frac{\left|G\right|}{\chi\left(1\right)}\right)^{2k-1}\cdot\delta_{\chi\in\mathbb{R}}
\]
(either by computing directly or by (\ref{eq:conv-disj})). This shows
that once more $N_{w}\equiv N_{w'}$ iff all characters are real,
and this does not seem to be as simple to show directly.
\item Similarly, applying (\ref{eq:formula}) to $w=\left[a,b\right]d\left[a,c\right]d^{-1}$
and $w'=\left\{ a,b\right\} d\left\{ a,c\right\} d^{-1}$ gives
\[
N_{w}^{\chi}=\frac{\left|G\right|^{2}}{\chi\left(1\right)^{3}}\sum_{g\in G}\left|\chi\left(g\right)\right|^{4},\quad N_{w'}^{\chi}=\frac{\left|G\right|^{2}}{\chi\left(1\right)^{3}}\sum_{g\in G}\left(\chi\left(g\right)\right)^{4}
\]
which shows that $N_{w}\equiv N_{w'}$ iff $\chi\left(g\right)\in\mathbb{R}\cup i\mathbb{R}$
for every $\chi\in\mathrm{Irr}\left(G\right)$ and $g\in G$. We leave
it as an exercise to find, for any $k\in\mathbb{N}$, words $w,w'$
which induce the same distribution on $G$ iff $\chi\left(g\right)^{k}\in\mathbb{R}$
for all $\chi\in\mathrm{Irr}\left(G\right)$ and $g\in G$.
\item For a word with dismissible letters for which $r\left(w\right)=1$,
we obtain 
\[
N_{w}^{\chi}=\frac{\left|G\right|^{n-1}}{\chi\left(1\right)^{n}}\sum_{\overline{x}\in G^{d}}\overline{\chi}\left(W_{1}\right)=\frac{\left|G\right|^{n}}{\chi\left(1\right)^{n}}N_{W_{1}}^{\chi}.
\]
Since $r=1$, $n$ is even (see Theorem \ref{thm:the-theorem}), and
we obtain by (\ref{eq:conv-disj}) 
\[
N_{w}^{\chi}=N_{\left[y_{1},y_{2}\right]\cdot\ldots\cdot\left[y_{n-1},y_{n}\right]\cdot W_{1}}^{\chi}.
\]
In fact, one can show that $w$ and $\left[y_{1},y_{2}\right]\cdot\ldots\cdot\left[y_{n-1},y_{n}\right]\cdot W_{1}$
are equivalent under $\mathrm{Aut}\left(\mathbf{F}_{d}\right)$. 
\item More generally, the formula (\ref{eq:formula}) can be interpreted
in the context of tuples of word maps, as follows: for $w_{1},\ldots,w_{r}\in\mathbf{F}_{d}$,
we observe the tuple $\overset{{\scriptscriptstyle \rightarrow}}{\boldsymbol{w}}=\left(w_{1},\ldots,w_{r}\right)$
as a word map $\overset{{\scriptscriptstyle \rightarrow}}{\boldsymbol{w}}:G^{d}\rightarrow G^{r}$,
defined by 
\[
\overset{{\scriptscriptstyle \rightarrow}}{\boldsymbol{w}}\left(x_{1},\ldots,x_{d}\right)=\left(w_{1}\left(\overline{x}\right),\ldots,w_{r}\left(\overline{x}\right)\right),
\]
and define 
\[
N_{\overset{{\scriptscriptstyle \rightarrow}}{\boldsymbol{w}}}\left(g_{1},\ldots,g_{r}\right)=\left|\left(\overset{{\scriptscriptstyle \rightarrow}}{\boldsymbol{w}}\right)^{-1}\left(\overline{g}\right)\right|=\sharp\left\{ \overline{x}\in G^{d}\,\middle|\,\ \overset{{\scriptscriptstyle \rightarrow}}{\boldsymbol{w}}\left(\overline{x}\right)=\overline{g}\right\} .
\]
Though $N_{\overset{{\scriptscriptstyle \rightarrow}}{\boldsymbol{w}}}$
is not a class function on $G^{r}$, we can still observe the ``Fourier
coefficient'' of $\eta\in\mathrm{Irr}\left(G^{r}\right)$,
\[
N_{\overset{{\scriptscriptstyle \rightarrow}}{\boldsymbol{w}}}^{\eta}=\left\langle N_{\overset{{\scriptscriptstyle \rightarrow}}{\boldsymbol{w}}},\eta\right\rangle =\frac{1}{\left|G\right|^{r}}\sum_{\overline{g}\in G^{r}}N_{\overset{{\scriptscriptstyle \rightarrow}}{\boldsymbol{w}}}\left(\overline{g}\right)\overline{\eta}\left(\overline{g}\right).
\]
With this notation we find that if the dismissible letter formula
for a word $w$ produces $W_{1},\ldots,W_{r}$, then for every $\chi\in\mathrm{Irr}\left(G\right)$
\begin{align*}
N_{w}^{\chi} & =\frac{\left|G\right|^{n-1}}{\chi\left(1\right)^{n}}\sum_{\overline{x}\in G^{d}}\overline{\chi}\left(W_{1}\left(\overline{x}\right)\right)\cdot\ldots\cdot\overline{\chi}\left(W_{r}\left(\overline{x}\right)\right)\\
 & =\frac{\left|G\right|^{n+r-1}}{\chi\left(1\right)^{n}}\cdot\frac{1}{\left|G\right|^{r}}\sum_{\overline{g}\in G^{r}}N_{\overset{{\scriptscriptstyle \rightarrow}}{\mathbf{W}}}\left(\overline{g}\right)\cdot\overline{\chi}\left(g_{1}\right)\cdot\ldots\cdot\overline{\chi}\left(g_{r}\right)\\
 & =\frac{\left|G\right|^{n+r-1}}{\chi\left(1\right)^{n}}\left\langle N_{\overset{{\scriptscriptstyle \rightarrow}}{\mathbf{W}}},\chi^{\otimes r}\right\rangle \\
 & =\frac{\left|G\right|^{n+r-1}}{\chi\left(1\right)^{n}}N_{\overset{{\scriptscriptstyle \rightarrow}}{\mathbf{W}}}^{\chi^{\otimes r}}.
\end{align*}
where $\overset{{\scriptscriptstyle \rightarrow}}{\mathbf{W}}=\left(W_{1},\ldots,W_{r}\right)$.
\end{enumerate}

\section{\label{sec:Preliminaries}Preliminaries}

The following proposition lists some simple facts about the Fourier
expansions of word maps.
\begin{prop}
\label{prop:misc}Let $G$ be a finite group, $w\in\mathbf{F}_{d}$,
and \textup{$\chi\in\mathrm{Irr}\left(G\right)$. }
\begin{enumerate}
\item \textup{The coefficient of $\chi$ in $N_{w}$ is given by 
\[
N_{w}^{\chi}=\frac{1}{\left|G\right|}\sum_{\overline{x}\in G^{d}}\overline{\chi}\left(w\right)=\frac{1}{\left|G\right|}\sum_{\overline{x}\in G^{d}}\overline{\chi}\left(w\left(x_{1},\ldots,x_{d}\right)\right).
\]
}
\item The Fourier expansion of $N_{w^{-1}}$ is given by $N_{w^{-1}}^{\chi}=\overline{N_{w}^{\chi}}$.
\item If $w_{1},w_{2}\in\mathbf{F}_{d}$ are words with disjoint letters
then 
\begin{equation}
N_{w_{1}\cdot w_{2}}^{\chi}=\frac{\left|G\right|}{\chi\left(1\right)}N_{w_{1}}^{\chi}N_{w_{2}}^{\chi}.\label{eq:conv-disj}
\end{equation}

\end{enumerate}
\end{prop}
\begin{proof}
$ $
\begin{enumerate}
\item It is well known that $\mathrm{Irr}\left(G\right)$ is an orthonormal
base for the space of class functions on $G$, endowed with inner
product 
\[
\left\langle f_{1},f_{2}\right\rangle =\frac{1}{\left|G\right|}\sum_{g\in G}f_{1}\left(g\right)\overline{f_{2}\left(g\right)}.
\]
Thus,
\begin{align*}
N_{w}^{\chi} & =\left\langle N_{w},\chi\right\rangle =\frac{1}{\left|G\right|}\sum_{g\in G}N_{w}\left(g\right)\overline{\chi}\left(g\right)\\
 & =\frac{1}{\left|G\right|}\sum_{g\in G}\sum_{{\overline{x}\in G^{d}\;\mathrm{s.t.}\atop w\left(x_{1},\ldots,x_{d}\right)=g}}\overline{\chi}\left(g\right)=\frac{1}{\left|G\right|}\sum_{\overline{x}\in G^{d}}\overline{\chi}\left(w\left(x_{1},\ldots,x_{d}\right)\right).
\end{align*}

\item This follows from 
\[
N_{w^{-1}}\left(g\right)=\overline{N_{w^{-1}}\left(g\right)}=\overline{N_{w}\left(g^{-1}\right)}=\sum_{\chi\in\mathrm{Irr}\left(G\right)}\overline{N_{w}^{\chi}\chi\left(g^{-1}\right)}=\sum_{\chi\in\mathrm{Irr}\left(G\right)}\overline{N_{w}^{\chi}}\chi\left(g\right).
\]

\item We recall that the convolution of two class functions is defined by
\begin{align*}
\left(f_{1}*f_{2}\right)\left(g\right) & =\frac{1}{\left|G\right|}\sum_{h\in G}f_{1}\left(h\right)f_{2}\left(h^{-1}g\right)
\end{align*}
and that for $\chi,\psi\in\mathrm{Irr}\left(G\right)$ 
\begin{equation}
\psi*\chi=\frac{\delta_{\psi,\chi}}{\chi\left(1\right)}\cdot\chi\label{eq:irr-conv}
\end{equation}
(see e.g. \cite[5.20]{Fol1995}). Using this, we see that 
\begin{align*}
N_{w_{1}\cdot w_{2}}\left(g\right) & =\sum_{h\in G}N_{w_{1}}\left(h\right)N_{w_{2}}\left(h^{-1}g\right)\\
 & =\sum_{h\in G}\left(\sum_{\chi\in\mathrm{Irr}\left(G\right)}N_{w_{1}}^{\chi}\cdot\chi\left(h\right)\right)\cdot\left(\sum_{\psi\in\mathrm{Irr}\left(G\right)}N_{w_{2}}^{\psi}\cdot\psi\left(h^{-1}g\right)\right)\\
 & =\sum_{\chi,\psi\in\mathrm{Irr}\left(G\right)}N_{w_{1}}^{\chi}\cdot N_{w_{2}}^{\psi}\cdot\left(\sum_{h\in G}\chi\left(h\right)\psi\left(h^{-1}g\right)\right)\\
 & =\sum_{\chi,\psi\in\mathrm{Irr}\left(G\right)}N_{w_{1}}^{\chi}\cdot N_{w_{2}}^{\psi}\cdot\left|G\right|\cdot\left(\chi*\psi\right)\left(g\right)\\
 & =\sum_{\chi,\psi\in\mathrm{Irr}\left(G\right)}N_{w_{1}}^{\chi}\cdot N_{w_{2}}^{\psi}\cdot\frac{\left|G\right|}{\chi\left(1\right)}\cdot\delta_{\chi,\psi}\chi\left(g\right)\\
 & =\sum_{\chi\in\mathrm{Irr}\left(G\right)}\frac{\left|G\right|}{\chi\left(1\right)}N_{w_{1}}^{\chi}N_{w_{2}}^{\chi}\chi\left(g\right).
\end{align*}

\end{enumerate}
\end{proof}
The next proposition already appears in \cite{AV10} (as Lemma $4.6$),
and it was also shown in (\ref{eq:[w,y]-from-formula}) to follow
from the formula for dismissible letters (\ref{eq:formula}). Nevertheless,
we present here two proofs for it. The first is interesting in its
own rights, and the second will be generalized to the proof of the
formula (\ref{eq:formula}), and serves as a warm-up. 
\begin{prop}
\label{prop:[w,z]}If $y$ is a letter which does not appear in $w$,
then 
\[
N_{\left[w,y\right]}^{\chi}=\frac{\left|G\right|}{\chi\left(1\right)}\left\langle N_{w}\chi,\chi\right\rangle =\frac{\left|G\right|}{\chi\left(1\right)}\sum_{\psi\in\mathrm{Irr}\left(G\right)}\left\langle \psi\chi,\chi\right\rangle N_{w}^{\psi}.
\]
\end{prop}
\begin{proof}[First proof]
The column orthogonality of characters states that
\[
\sum_{\chi\in\mathrm{Irr}\left(G\right)}\chi\left(g\right)\overline{\chi}\left(h\right)=\left|C_{G}\left(g\right)\right|\delta_{g\sim h}\qquad\mathrm{where}\qquad\delta_{g\sim h}=\begin{cases}
1 & g^{G}=h^{G}\\
0 & \mathrm{else}
\end{cases}.
\]
Thus,
\begin{align*}
N_{\left[w,y\right]}\left(g\right) & =\sum_{x\in G}N_{w}\left(x\right)\cdot\sharp\left\{ y\in G|\ \left[x,y\right]=g\right\} \\
 & =\sum_{x\in G}N_{w}\left(x\right)\cdot\sharp\left\{ y\in G|\ yx^{-1}y^{-1}=x^{-1}g\right\} \\
 & =\sum_{x\in G}N_{w}\left(x\right)\cdot\left|C_{G}\left(x^{-1}\right)\right|\cdot\delta_{x^{-1}\sim x^{-1}g}\\
 & =\sum_{x\in G}N_{w}\left(x\right)\sum_{\chi\in\mathrm{Irr}\left(G\right)}\chi\left(x^{-1}\right)\overline{\chi}\left(x^{-1}g\right)\\
 & =\sum_{x\in G}N_{w}\left(x\right)\sum_{\chi\in\mathrm{Irr}\left(G\right)}\overline{\chi}\left(x^{-1}\right)\chi\left(x^{-1}g\right)\qquad\left(\mathrm{since}\: N_{\left[w,y\right]}\left(g\right)\in\mathbb{R}\right)\\
 & =\sum_{\chi\in\mathrm{Irr}\left(G\right)}\sum_{x\in G}N_{w}\left(x\right)\chi\left(x\right)\chi\left(x^{-1}g\right)\\
 & =\left|G\right|\sum_{\chi\in\mathrm{Irr}\left(G\right)}\left(N_{w}\chi*\chi\right)\left(g\right).
\end{align*}
By (\ref{eq:irr-conv}) and linearity, any class function $f$ and
$\chi\in\mathrm{Irr}\left(G\right)$ satisfy 
\[
f*\chi=\frac{\left\langle f,\chi\right\rangle }{\chi\left(1\right)}\chi,
\]
and for $f=N_{w}\chi$ this gives 
\[
N_{\left[w,y\right]}\left(g\right)=\sum_{\chi\in\mathrm{Irr}\left(G\right)}\frac{\left|G\right|}{\chi\left(1\right)}\left\langle N_{w}\chi,\chi\right\rangle \chi\left(g\right).
\]

\end{proof}

\begin{proof}[Second proof]
We shall make use of the following classical theorem about compact
groups:
\begin{fact*}[Peter-Weyl]
Let $\left\{ \rho_{s}\right\} $ be unitary representatives for the
isomorphism classes of the irreducible representations of $G$. Let
$1\le i,j\le\dim\left(\rho_{s}\right)$ and $1\le k,l\le\dim\left(\rho_{s'}\right)$.
Then 
\begin{equation}
\frac{1}{\left|G\right|}\sum_{g\in G}\rho_{s}\left(g\right)_{i,j}\overline{\rho_{s'}\left(g\right)_{k,l}}=\frac{\delta_{s,s'}\delta_{i,k}\delta_{j,l}}{\dim\left(\rho_{s}\right)}.\label{eq:peter-weyl}
\end{equation}
\textup{}
\end{fact*}
We assume $w\in\mathbf{F}_{d}=\mathbf{F}\left(x_{1},\ldots,x_{d}\right)$,
$\left[w,y\right]\in\mathbf{F}_{d+1}=\mathbf{F}\left(x_{1},\ldots,x_{d},y\right)$.
Note that $\left[w,y\right]$ and $\left[w,y\right]^{-1}=\left[y,w\right]$
are equivalent under $\mathrm{Aut}\left(\mathbf{F}_{d+1}\right)$,
so by Proposition \ref{prop:misc} $\left(2\right)$ all Fourier coefficients
of $\left[w,y\right]$ are real. Let $\rho$ be a unitary representation
which affords the character $\chi$. By (\ref{eq:inversion-formula})
we have 
\begin{align*}
N_{\left[w,y\right]}^{\chi}=\overline{N_{\left[w,y\right]}^{\chi}} & =\frac{1}{\left|G\right|}\sum_{{\overline{x}\in G^{d}\atop y\in G^{\hphantom{d}}}}\chi\left(wyw^{-1}y^{-1}\right)=\frac{1}{\left|G\right|}\sum_{{\overline{x}\in G^{d}\atop y\in G^{\hphantom{d}}}}\sum_{i=1}^{\dim\rho}\rho\left(wyw^{-1}y^{-1}\right)_{i,i}\\
 & =\frac{1}{\left|G\right|}\sum_{{\overline{x}\in G^{d}\atop y\in G^{\hphantom{d}}}}\sum_{i,j,k,\ell=1}^{\dim\rho}\rho\left(w\right)_{i,j}\rho\left(y\right)_{j,k}\rho\left(w^{-1}\right)_{k,\ell}\rho\left(y^{-1}\right)_{\ell,i}\\
 & =\frac{1}{\left|G\right|}\sum_{{\overline{x}\in G^{d}\atop y\in G^{\hphantom{d}}}}\sum_{i,j,k,\ell=1}^{\dim\rho}\rho\left(w\right)_{i,j}\rho\left(y\right)_{j,k}\overline{\rho\left(w\right)_{\ell,k}\rho\left(y\right)_{i,\ell}}\\
 & \overset{\left(*\right)}{=}\frac{1}{\dim\rho}\sum_{\overline{x}\in G^{d}}\sum_{i,j,k,\ell=1}^{\dim\rho}\rho\left(w\right)_{i,j}\overline{\rho\left(w\right)_{\ell,k}}\delta_{i,j}\delta_{k,\ell}\\
 & =\frac{1}{\dim\rho}\sum_{\overline{x}\in G^{d}}\sum_{i,k=1}^{\dim\rho}\rho\left(w\right)_{i,i}\overline{\rho\left(w\right)_{k,k}}\\
 & =\frac{1}{\chi\left(1\right)}\sum_{\overline{x}\in G^{d}}\chi\left(w\right)\overline{\chi\left(w\right)}\\
 & =\frac{1}{\chi\left(1\right)}\sum_{g\in G}N_{w}\left(g\right)\chi\left(g\right)\overline{\chi\left(g\right)}\\
 & =\frac{\left|G\right|}{\chi\left(1\right)}\left\langle N_{w}\chi,\chi\right\rangle 
\end{align*}
where $\left(*\right)$ is by Peter-Weyl.
\end{proof}

\section{\label{sec:Proof-of-the}Proofs of the formulas}

We begin by the formula for square letters.
\begin{thm}
Let $w=w_{1}yw_{2}yw_{3}\in\mathbf{F}\left(x_{1},\ldots,x_{d},y\right)$
with $y$ a square, i.e. $w_{1},w_{2},w_{3}\in\mathbf{F}\left(x_{1},\ldots,x_{d}\right)$.
Then
\[
N_{w}^{\chi}=\frac{\left|G\right|}{\chi\left(1\right)}\mathcal{FS}_{\chi}N_{w_{1}w_{2}^{-1}w_{3}}^{\chi},
\]
where $\mathcal{FS}_{\chi}$ is the Frobenius-Schur indicator (\ref{eq:FS_indicator}).\end{thm}
\begin{proof}
First, a word is always equivalent to its cyclic shifts under $\mathrm{Aut}\left(\mathbf{F}_{d}\right)$,
so that $N_{w}\equiv N_{yw_{2}yw_{3}w_{1}}$. Now applying the automorphism
$x_{i}\mapsto x_{i}$, $y\mapsto yw_{2}^{-1}$, we obtain $N_{w}\equiv N_{y^{2}w_{2}^{-1}w_{3}w_{1}}$.
Using (\ref{eq:conv-disj}) we have
\[
N_{w}^{\chi}=N_{y^{2}w_{2}^{-1}w_{3}w_{1}}^{\chi}=\frac{\left|G\right|}{\chi\left(1\right)}N_{y^{2}}^{\chi}N_{w_{2}^{-1}w_{3}w_{1}}^{\chi}=\frac{\left|G\right|}{\chi\left(1\right)}\mathcal{FS}_{\chi}N_{w_{1}w_{2}^{-1}w_{3}}^{\chi}
\]
where we recall that $\mathcal{FS}_{\chi}=N_{y^{2}}^{\chi}$ \cite[Proposition 39]{serre1977linear}.
\end{proof}
We move on to the formula for dismissible letters. We take $w\in\mathbf{F}\left(x_{1},\ldots,x_{d},y_{0},\ldots,y_{n-1}\right)$
to be a word with the $y_{i}$ dismissible, and assume (by applying
a cyclic shift if necessary) that $w$ ends in some $y_{i}^{\pm}$.
\begin{thm}
\label{thm:the-theorem}Let $w_{0},\ldots,w_{2n-1}\in F\left(x_{1},\ldots,x_{d}\right)$,
and $y_{0},\ldots,y_{n-1}$ be new letters. Let $z_{0},\ldots,z_{2n-1}$
be a permutation of $y_{0},\ldots,y_{n-1},y_{0}^{-1},\ldots,y_{n-1}^{-1}$,
and let 
\[
w=w_{0}z_{0}w_{1}z_{1}\cdots w_{2n-1}z_{2n-1}.
\]
Define a permutation $\tau\in Sym\left\{ 0,\ldots,2n-1\right\} $
by $\tau\left(i\right)=j$ iff $z_{i}=z_{j}^{-1}$, and $\sigma\in Sym$$\left\{ 0,\ldots,2n-1\right\} $
by $\sigma=\left(0\:1\:2\ldots\:2n\!-\!1\right)\circ\tau$ (i.e. $\sigma\left(k\right)=\tau\left(k\right)+1\ \left(\mathrm{mod}\ 2n\right)$).
Let $\sigma=\prod_{s=1}^{r}\left(\sigma_{1}^{s}\:\sigma_{2}^{s}\ldots\:\sigma_{m_{s}}^{s}\right)$
be the decomposition of $\sigma$ into disjoint cycles. Then $r\not\equiv n\left(\mathrm{mod}\,2\right)$,
and the Fourier expansion of $N_{w}$ is given by 
\[
N_{w}^{\chi}=\frac{\left|G\right|^{n-1}}{\chi\left(1\right)^{n}}\sum_{\left(x_{1},\ldots,x_{d}\right)\in G^{d}}\Bigg(\prod_{i=1}^{r}\overline{\chi}\Bigg(\underbrace{\prod\nolimits _{j=1}^{m_{i}}w_{\sigma_{j}^{i}}}_{W_{i}}\Bigg)\Bigg).
\]
\end{thm}
\begin{proof}
We strongly advise reading the second proof of Proposition \ref{prop:[w,z]}
before going any further, as it is much more accessible and contains
the main ideas of the proof. 

First, we observe that 
\[
\left(-1\right)^{r}=\left(-1\right)^{2n-r}=\sgn\sigma=-\sgn\tau=-\left(-1\right)^{n}
\]
and thus $r\not\equiv n\left(\mathrm{mod}\,2\right)$. We shall write
$i\boxplus j$ for $\left(i+j\right)\mathrm{mod}\,2n$. We shall use
the fact, already exploited in the second proof of Proposition \ref{prop:[w,z]},
that for $\rho\in\widehat{G}$ and $1\le i,j,k,l\le\dim\rho$, 
\begin{equation}
\sum_{x\in G}\rho\left(y\right)_{i,j}\rho\left(y^{-1}\right)_{k,l}=\frac{\left|G\right|}{\dim\rho}\delta_{i,l}\delta_{j,k},\label{eq:lemma_delta}
\end{equation}
(this follows from $\rho\left(y^{-1}\right)_{k,l}=\overline{\rho\left(y\right)_{l,k}}$
and the Peter-Weyl Theorem). By (\ref{eq:inversion-formula}) we have
\begin{align*}
\overline{N_{w}^{\chi}\left(g\right)} & =\frac{1}{\left|G\right|}\sum_{{x_{1},\ldots,x_{d}\in G\atop y_{0},\ldots,y_{n-1}\in G}}\chi\left(\prod_{j=0}^{2n-1}w_{j}z_{j}\right)=\frac{1}{\left|G\right|}\sum_{{x_{1},\ldots,x_{d}\in G\atop y_{0},\ldots,y_{n-1}\in G}}\sum_{i=1}^{\dim\rho}\rho\left(\prod_{j=0}^{2n-1}w_{j}z_{j}\right)_{i,i}=\\
 & =\frac{1}{\left|G\right|}\sum_{{x_{1},\ldots,x_{d}\in G\atop y_{0},\ldots,y_{n-1}\in G}}\sum_{{k_{0},\ldots,k_{2n-1}=1\atop l_{0},\ldots,l_{2n-1}=1}}^{\dim\rho}\prod_{j=0}^{2n-1}\rho\left(w_{j}\right)_{l_{j},k_{j}}\prod_{j=0}^{2n-1}\rho\left(z_{j}\right)_{k_{j},l_{j\boxplus1}}\\
 & =\frac{\left|G\right|^{n-1}}{\chi\left(1\right)^{n}}\sum_{x_{1},\ldots,x_{d}\in G}\sum_{{k_{0},\ldots,k_{2n-1}=1\atop l_{0},\ldots,l_{2n-1}=1}}^{\dim\rho}\prod_{j=0}^{2n-1}\rho\left(w_{j}\right)_{l_{j},k_{j}}\prod_{j=0}^{2n-1}\delta_{k_{j},l_{\tau\left(j\right)\boxplus1}}\delta_{l_{j\boxplus1},k_{\tau\left(j\right)}}\quad\left(\mathrm{by\:}\eqref{eq:lemma_delta}\right)\\
 & =\frac{\left|G\right|^{n-1}}{\chi\left(1\right)^{n}}\sum_{x_{1},\ldots,x_{d}\in G}\sum_{{k_{0},\ldots,k_{2n-1}=1\atop l_{0},\ldots,l_{2n-1}=1}}^{\dim\rho}\prod_{j=0}^{2n-1}\rho\left(w_{j}\right)_{l_{j},k_{j}}\prod_{j=0}^{2n-1}\delta_{k_{j},l_{\sigma\left(j\right)}}\quad\left({\mathrm{since}\;\tau\left(j\right)\boxplus1=\sigma\left(j\right)\atop \mathrm{and}\;\tau=\tau^{-1}}\right)\\
 & =\frac{\left|G\right|^{n-1}}{\chi\left(1\right)^{n}}\sum_{x_{1},\ldots,x_{d}\in G}\sum_{l_{0},\ldots,l_{2n-1}=1}^{\dim\rho}\prod_{j=0}^{2n-1}\rho\left(w_{j}\right)_{l_{j},l_{\sigma\left(j\right)}}\\
 & =\frac{\left|G\right|^{n-1}}{\chi\left(1\right)^{n}}\sum_{x_{1},\ldots,x_{d}\in G}\sum_{l_{0},\ldots,l_{2n-1}=1}^{\dim\rho}\prod_{s=1}^{r}\prod_{t=1}^{m_{s}}\rho\left(w_{\sigma_{t}^{s}}\right)_{l_{\sigma_{t}^{s}},l_{\sigma\left(\sigma_{t}^{s}\right)}}\quad\left(\mathrm{since}\;\left\{ \sigma_{t}^{s}\right\} _{{s=1..r\atop t=1..m_{s}}}=\left\{ 1\ldots n\right\} \right)\\
 & =\frac{\left|G\right|^{n-1}}{\chi\left(1\right)^{n}}\sum_{x_{1},\ldots,x_{d}\in G}\sum_{l_{0},\ldots,l_{2n-1}=1}^{\dim\rho}\prod_{s=1}^{r}\prod_{t=1}^{m_{s}}\rho\left(w_{\sigma_{t}^{s}}\right)_{l_{\sigma_{t}^{s}},l_{\sigma_{t+1\left(\mathrm{mod}\, m_{s}\right)}^{s}}}\\
 & =\frac{\left|G\right|^{n-1}}{\chi\left(1\right)^{n}}\sum_{x_{1},\ldots,x_{d}\in G}\sum_{l_{\sigma_{0}^{1}},l_{\sigma_{0}^{2}},\ldots,l_{\sigma_{0}^{r}}=1}^{\dim\rho}\prod_{s=1}^{r}\rho\left(\prod_{t=1}^{m_{s}}w_{\sigma_{t}^{s}}\right)_{l_{\sigma_{0}^{s}},l_{\sigma_{0}^{s}}}\\
 & =\frac{\left|G\right|^{n-1}}{\chi\left(1\right)^{n}}\sum_{x_{1},\ldots,x_{d}\in G}\prod_{s=1}^{r}\chi\left(\prod_{t=1}^{m_{s}}w_{\sigma_{t}^{s}}\right).
\end{align*}

\end{proof}
\bibliographystyle{amsalpha}
\bibliography{/home/math/parzan/Math/BibTex}

\end{document}